\documentclass[11pt]{article}
\usepackage{geometry}
\usepackage[latin1]{inputenc}
\usepackage[italian,english]{babel}
\usepackage{amsmath, amsfonts, amsthm}
\usepackage{paralist, times, indentfirst}
\usepackage[T1]{fontenc}
\usepackage{amsthm}
\newcommand{\R}{\mathbb{R}}

\newcommand{\ABR}{(Abramowitz and Stegun, 1965)}

\newcommand{\BBMP}{(Betta et al., 1999)}

\newcommand{\FK}{(Freitas and Krej\u{c}i\u{r}\'{i}k, 2015)}
\newcommand{\AHEx}{(Henrot, 2006)}

\newcommand{\LOS}{(Lacey et al., 1998)}

\newtheorem{proposition}{Proposition}

\newtheorem{theorem}{Theorem}

\theoremstyle{remark}

{\left\lbrace\begin{array}{@{}l@{}}}%
{\end{array}\right.}
\title{Some Remarks on the Robin-Laplacian Eigenvalues}

\begin{document}

\vspace*{45 mm}
\begin{center}
{\bf \fontsize{13}{13}\selectfont 
Some Remarks on Robin-Laplacian Eigenvalues}\\
\vspace{9pt}
Leonardo Trani \footnote{Dipartimento di Matematica e Applicazioni \lq\lq R. Caccioppoli\rq\rq, Universit\`a di Napoli \mbox{Federico II}, Via Cintia, 80126 Napoli, Italy. $\mathtt{leonardo.trani@unina.it}$}\\

\end{center}
\vspace{11pt}
{\fontsize{10}{10}\selectfont \noindent \textbf{Abstract} - We study some properties of Laplacian eigenvalues with negative Robin boundary conditions. We will show some monotonicity properties on annuli of the first eigenvalue by means of shape optimization techniques.}\\

{\fontsize{10}{10}\selectfont \noindent \textbf{Riassunto} - In questa nota si studiano gli autovalori del Laplaciano con condizioni al bordo di Robin negative. Mostreremo alcune proprietà di monotonia per il primo autovalore sull'anello attraverso l'uso di tecniche relative all'ottimizzazione di forma.}\\

{\fontsize{10}{10}\selectfont \noindent \textbf{Keywords:} Robin-Laplacian Eigenvalues, Shape Derivative, Spectral Inequality}\\

{\fontsize{10}{10}\selectfont \noindent \textbf{Mathematics Subject Classification (2010):} 35P15, 35J25}
\vspace{11pt}

\section{\fontsize{10}{10}\selectfont \bf 1 -  INTRODUCTION}
We consider the following eigenvalue problem
\begin{equation}\label{RobinLaplacianProblem}
\begin{cases}
-\Delta u=\lambda_i u & \mbox{in}\ \Omega\\
\frac{\partial u}{\partial\nu}=\alpha u & \mbox{on}\ \partial\Omega
\end{cases}
\end{equation}
where $\alpha > 0$ and
we investigate the monotonicity of the first eigenvalue $\lambda_1$ in the annulus, defined as $A_{r_1,r_2}=B_{r_2}\setminus{\overline{B}}_{r_1}$ for $r_1<r_2$, where $B_r$ is the open ball of radius $r$, following \FK, with respect to $r_2$.  We prove the following
\begin{theorem}\label{TH1}
Let $V_1$ be the following vectorial field in $\mathbb{R}^2$ 
\begin{equation}
V_1(x)= \begin{cases}
        \nu & \mbox{if}\ \vert x\vert=r_2\\
        0 & \mbox{otherwise}
        \end{cases}
\end{equation}
where $\nu$ is the unit outward normal vector of $\partial\Omega$, then 
$$
d\lambda_1\left(A_{r_1,r_2},V_1\right)>0\mbox{.}
$$
In particular, if $r_2<\tilde{r_2}$ than
$$
\lambda_1\left(A_{r_1,r_2}\right)<\lambda_1\left(A_{r_1,\tilde{r_2}}\right)\mbox{.}
$$
\end{theorem}
%!!!!!- ATTENZIONE: CONTROLLA I CONTI DI QUESTA PARTE -!!!!!
%An equivalent stetment will be obteined for the inner radius, indeed, when we fix the external radius and the inner one decreases, the first eigenvalue of the problem (\ref{RobinLaplacianProblem}) increases. The way to prove this last statement is the same of the Theorem \ref{TH1}.\\

%After the proof of the above theorem, we have remarked a fact: when we pass from the ball to an annulus, with the same volume and a very little inner hole, the first eigenvalue becomes smaller.

On the other hand, we have observed that, when the parameter $\alpha=\sqrt[d]{\frac{\omega_d}{\vert\Omega\vert}}$, where
$$
\omega_d=\displaystyle{\frac{\pi^{\frac{d}{2}}}{\Gamma\left(1+\frac{d}{2}\right)}},
$$
problem (\ref{RobinLaplacianProblem}) on the ball is equivalent to the Stekloff-Laplacian problem, for which is known the value of the first non-trivial eigenvalue.
The statement is the following
\begin{theorem}\label{TH2}
Let $\Omega$ be a bounded open subset of $\R^d$ with Lipschitz boundary and let $B_r$ be the ball with the same measure as $\Omega$, that is $r = \sqrt[d]{\frac{\vert\Omega\vert}{\omega_d}}$.When $\alpha=\frac{1}{r}$ the following inequality holds
\begin{equation}\label{iso1/R}
\lambda_2(\Omega)\leq\lambda_2(B_r)=0.
\end{equation}
\end{theorem}
\vspace{11pt}

\section{\fontsize{10}{10}\selectfont \bf 2 - NOTATIONS AND PRELIMINARIES}
In this section we recall some properties of the eigenvalues of the Laplacian. Let $\Omega$ be a bounded open subset of $\R^d$, the eigenvalues of (\ref{RobinLaplacianProblem}) form a sequence $\lambda_1\leq\lambda_2\leq\ldots\leq\lambda_m\leq\ldots$ such that $\lambda_m\rightarrow\infty$, and they can be characterized with \textit{min-max} formulation, that is %trovare la fonte di questa cosa
\begin{equation}\label{minmax}
\lambda_m(\Omega)=\inf_{{E_m\subset H^1(\Omega)}_{ \dim E_m=m}}\left(\max_{v\in E_m\setminus\{0\}} \frac{\displaystyle\int_\Omega{\vert\nabla v\vert}^2dx-\alpha\displaystyle\int_{\partial\Omega}v^2d\sigma}{\displaystyle\int_\Omega v^2dx}\right).
\end{equation}
In particular, the first one is given by
\begin{equation}\label{max_l1}
\lambda_1(\Omega)=\inf_{v\in H^1(\Omega)\setminus\{0\}}\frac{\displaystyle\int_\Omega{\vert\nabla v\vert}^2dx-\alpha\displaystyle\int_{\partial\Omega}v^2d\sigma}{\displaystyle\int_\Omega v^2dx}.
\end{equation}
Using the costant as test function in the Rayleigh quotient (\ref{max_l1}), we obtain the following inequality, which allows to see that $\lambda_1(\Omega)<0$:
\begin{equation}\label{neg_eig}
\lambda_1(\Omega)\leq-\alpha\frac{\sigma(\partial\Omega)}{\vert\Omega\vert}, 
\end{equation}
where $\sigma(\Omega)$ stands for the $(d-1)$-dimensional Hausdorff measure of $\partial\Omega$ and $\vert\Omega\vert$ stands for the Lebesgue measure of $\Omega$. The above inequality implies that the first eigenvalue is not bounded from below when the volume is fixed.
As in \LOS one can see that the first eigenvalue is simple and has a positive associated eigenfunction.\\
Having in mind this fact, we obtain that the associated eigenfunction to problem (\ref{RobinLaplacianProblem}) on the annulus is radial, and then we can write problem (\ref{RobinLaplacianProblem}) as follows
\begin{equation}\label{radpr}
\begin{cases}
-\frac{1}{r^{d-1}}\left[r^{d-1}\phi^{\prime}(r)\right]^\prime&=\lambda_1(A_{r_1,r_2})\phi(r), \ r_1<r<r_2\\
-\phi^\prime(r_1)-\alpha\phi(r_1)&=0\\
\phi^\prime(r_2)-\alpha\phi(r_2)&=0
\end{cases}
\end{equation}
where $u_1(x)=\phi(\vert x\vert)$ is the first eigenfunction in $A_{r_1,r_2}$. The solutions of (\ref{radpr}) are given by
\begin{equation}\label{ei_fu}
\phi(r)=r^{-p}\left[C_1K_p(\sqrt{\lambda_1(A_{r_1,r_2})}r)+C_2I_p(\sqrt{\lambda_1(A_{r_1,r_2})}r)\right]\mbox{,}
\end{equation}
where $C_1$ and $C_2$ are implicite defined by the boundary conditions as in \FK, and where the functions $I_p$ and $K_p$ are modified Bessel functions of order $p$, see for istance \ABR, and
$$
p=\frac{d-2}{2}\mbox{.}
$$
For a long time, it was conjectured that balls maximize $\lambda_1$ among bounded open sets with given volume.  
Only recently, in \FK, the authors disprove such conjecture by showing that there exists an annulus, for which $\vert A_{r_1,r_2}\vert=\vert B_r\vert$ such that
$$
\lambda_1(A_{r_1,r_2})>\lambda_1(B_r)
$$
for $\alpha$ suitable great.
More precisely, they prove the following asymptotics for $\lambda_1$:
\begin{eqnarray}\label{asymp}
\lambda_1(A_{r_1,r_2})=-\alpha^2-\frac{\alpha}{r_2}+o(\alpha)\\
\lambda_1(B_r)=-\alpha^2-\frac{\alpha}{r}+o(\alpha)\mbox{.}\nonumber
\end{eqnarray}
In order to prove Theorem 1, we need to recall the classical Hadamard formula for $\lambda_1$, which is
\begin{equation}\label{derWagner}
d\lambda_1(\Omega, V)=\int_{\partial\Omega}{\left({\vert\nabla u_1\vert}^2-\lambda_1(\Omega)u_1^2-2\alpha^2u_1^2-\alpha Hu_1^2\right)(V\cdot \nu)}d\sigma
\end{equation}
where $\Omega\subset\mathbb{R}^2$ is smooth, $H$ is the mean curvature at a point $x$ of $\partial\Omega$, $\nu$ is the unit outward normal vector of boundary $\partial\Omega$ and $V$ is a smooth vector field defined on $\partial\Omega$.\\ %Posso scriverla meglio, magari introducendo che cosa sono i domini perturbati
For the proof of Theorem \ref{TH2}, we need the following weighted isoperimetric inequality from \BBMP :
\begin{theorem}
Let $\Omega$ be a bounded open subset of $\R^d$ with Lipschitz boundary, $B_r$ a ball, such that $\vert\Omega\vert=\vert B_r\vert$, and $\psi:[0,+\infty)\rightarrow[0,+\infty)$ a non-decreasing function such that
$$ 
\left(\psi(t^{\frac{1}{d}})-\psi(0)\right)t^{1-\frac{1}{d}}
$$
is convex for every $t\geq 0$ 
\begin{equation}\label{BrockMercaldo}
\int_{\partial\Omega}\psi(\vert x\vert)d\sigma\geq\int_{\partial B_r}\psi(\vert x\vert)d\sigma.
\end{equation}

\end{theorem}
Another important remark in order to prove the Theorem \ref{TH2} is about the  eigenvalues of the Stekloff-Laplacian problem,
\begin{equation}\label{StekloffProblem}
\begin{cases}
-\Delta u= 0 & \mbox{in}\ \Omega\\
\frac{\partial u}{\partial\nu}=p_i u & \mbox{on}\ \partial\Omega
\end{cases}.
\end{equation}
where $\Omega$ is a bounded open set with Lipschitz boundary. The eigenvalues of (\ref{StekloffProblem}) form a sequence $0=p_1\leq p_2\leq\ldots\leq p_m\leq\ldots$ and they can be characterized, like in \AHEx , with the variational formulation
\begin{equation}
p_m(\Omega)=\min_{v\in H^1(\Omega)\setminus\{0\}}\left\{\frac{\displaystyle\int_{\Omega}\vert\nabla v\vert^2 dx}{\displaystyle\int_{\partial\Omega} v^2 d\sigma}\ :\  \int_{\partial\Omega} v u_id\sigma =0,\ i=1,\ldots,m-1 \right\}\mbox{,}
\end{equation}
where $u_i$ is the eigenfunction associated to the eigenvalue $p_i(\Omega)$.\\
It is known that are $p_2(B_r)=p_3(B_r)=\ldots=p_{d+1}(B_r)=\frac{1}{r}$ and the associated eigenfunctions are $\zeta_i(x)=x_{i-1}$ with $i=2,\ldots , d+1$. For that reason, choosing in problem (\ref{RobinLaplacianProblem}) $\alpha=p_2(B_r)=\frac{1}{r}$, we obtain $\lambda_2(B_r)=\lambda_3(B_r)=\ldots=\lambda_{d+1}(B_r)=0$. %Indeed, when we consider the unit ball $B_1$ and $\alpha =1$, we have $\lambda_2(B_1)=0$ and, by scaling property, we can extend this result to the ball $B_r$ when $\alpha=\frac{1}{r}$.\\

Before proceeding, applying (\ref{derWagner}) to the annulus in $\mathbb{R}^2$ of radii $r_1<r_2$, recalling that a volume preserving vector field is a smooth vector field $V:\Omega\subset\mathbb{R}^2\rightarrow\mathbb{R}^2$ such that

$$
\int_{\partial\Omega}(V,\nu)d\sigma=0\mbox{,}
$$

we obtain the following stationary condition:

\begin{proposition}

Let $A_{r_1,r_2}$ be an annulus of $\mathbb{R}^2$ and let $V$ be a volume preserving vector field in $A_{r_1,r_2}$, then

\begin{equation}\label{stacond}
\phi^2(r_2)\left(k^2 - \alpha^2 - \frac{\alpha}{r_2}\right)-\phi^2(r_1)\left(k^2 - \alpha^2 + \frac{\alpha}{r_1}\right)=0\Rightarrow d\lambda_1(A_{r_1,r_2}, V)=0
\end{equation}

where $\phi$ is the eigenfunction given in (\ref{ei_fu}), $k^2=-\lambda_1(A_{r_1,r_2})$ and $\alpha$ is the positive parameter in the Robin boundary condition.

\end{proposition}

\begin{proof} By (\ref{derWagner})
\begin{eqnarray*}
d\lambda_1\left(A_{r_1,r_2}, V\right) & = & \int_{\partial A_{r_1,r_2}}\left(\left\vert\nabla u\right\vert^2 + k^2u^2 - 2\alpha^2 u^2 - \alpha Hu^2\right)\left(V\cdot \nu\right)ds\\ 
& = & \left(k^2 - \alpha^2 -\frac{\alpha}{r_2}\right)\phi^2\left(r_2\right)\int_{\partial B_{r_2}}(V\cdot \nu)ds \\ 
& + & \left(k^2 -\alpha^2 + \frac{\alpha}{r_1}\right)\phi^2(r_1)\int_{\partial B{r_1}}(V\cdot \nu)ds\mbox{,} 
\end{eqnarray*}
and, having in mind that the vectorial field $V$ is volume preserving, or equivalently
$$\int_{\partial A_{r_1,r_2}}(V\cdot \nu)ds=0 \Rightarrow\int_{\partial B_{r_1}}(V\cdot \nu )ds=-\int_{\partial B_{r_2}}(V\cdot \nu)ds
$$
and then
\begin{eqnarray*}
d\lambda_1(A_{r_1,r_2}, V) & = & 
\left[\phi^2(r_2)\left(k^2 - \alpha^2 - \frac{\alpha}{r_2}\right)\right.\\
& - &\left.\phi^2(r_1)\left(k^2 - \alpha^2 + \frac{\alpha}{r_1}\right)\right]\displaystyle\int_{\partial B_{r_2}}(V\cdot \nu)ds\mbox{,} 
\end{eqnarray*}
which implies (\ref{stacond}).
\end{proof}

Let 
$$
G(r_2)=\phi^2(r_2)\left(k^2 - \alpha^2 - \frac{\alpha}{r_2}\right)-\phi^2(r_1)\left(k^2 - \alpha^2 + \frac{\alpha}{r_1}\right)\mbox{,}
$$
using the volume constraint ${r_2}^2-{r_1}^2=C$ and the boundary conditions in (\ref{radpr}), we obtain
\begin{eqnarray*}
\frac{dG}{dr_2}(r_2) & = & 2\alpha\phi^2\left(r_2\right)\left(k^2 - \alpha^2 -\frac{\alpha}{r_2}+\frac{1}{2{r_2}^2}\right)\\
& + &\frac{2\alpha\phi^2\left(r_1\right)r_2}{r_1}\left(k^2 -\alpha^2+\frac{\alpha}{r_1}+\frac{1}{2{r_1}^2}\right).
\end{eqnarray*}
Using the asymptotics (\ref{asymp}), we have
$$
\frac{dG}{dr_2}(r_2)=2\alpha\phi^2(r_2)\left(\frac{1}{2r_2^2}+o(\alpha)\right)+\frac{2\alpha\phi^2(r_1)r_2}{r_1}\left(\frac{\alpha}{r_2}+\frac{\alpha}{r_1}+\frac{1}{2r_1^2}+o(\alpha)\right)
$$
and $\frac{dG}{dr_2}(r_2)$ is positive for $\alpha$ greater than a critical value, said $\alpha_c$.

\vspace{11pt}

\section{\fontsize{10}{10}\selectfont \bf 3 - PROOF OF THEOREM 1}
When $d=2$ (\ref{radpr}) becomes
\begin{equation}\label{ODE1}
\begin{cases}
\phi^{\prime\prime}(r)+\displaystyle\frac{\phi^{\prime}(r)}{r}+\lambda \phi(r)=0\\ \phi^{\prime}(r_1)=-\alpha\phi(r_1)\\ 
\phi^{\prime}(r_2)=\alpha\phi(r_2)
\end{cases}
\end{equation}
where $\lambda = \lambda_1(A_{r_1,r_2})$.\\
From (\ref{derWagner}) we obtain
\begin{equation}\label{derV1}
d\lambda\left(A_{r_1,r_2},V_1\right)=2\pi r_2\phi^2\left(r_2\right)\left(-\lambda-\alpha^2-\frac{\alpha}{r_2}\right)
\end{equation}
and using (\ref{derV1})
we can prove the statement by proving that
$$
\left(\lambda+\alpha^2+\frac{\alpha}{r_2}\right) < 0\mbox{.}
$$

Setting $z=\displaystyle\frac{\phi^\prime}{\phi}$ (having in mind that $\phi>0$), using (\ref{ODE1}), we obtain that $z$ satisfies
\begin{equation}\label{ODE2}
\frac{dz}{dr}+z^2+\frac{z}{r}+\lambda=0\ \mbox{in}\ (r_1,r_2)
\end{equation}
and then
$$
\frac{dz}{dr}(r_2)=-\left(\lambda+\alpha^2+\frac{\alpha}{r_2}\right)\mbox{.}
$$
From the boundary conditions in (\ref{ODE1}) we have $z(r_1)=-\alpha$ and $z(r_2)=\alpha$. Then defining

\begin{equation}\label{csi}
\xi=\sup\left\{\rho\in(r_1,r_2): z(\rho)<0\right\},
\end{equation}

we have that $\xi<r_2$ and $z(\xi)=0$, and using (\ref{ODE2}) we obtain that

\begin{equation}\label{dz/drxi>0}
\frac{dz}{dr}(\xi)=-\lambda >0.
\end{equation}
Our aim is to prove that $\frac{dz}{dr}(r_2)>0$. Let $\xi_1$ define by 
\begin{equation}\label{csi1}
\xi_1=\sup\left\{\rho\in (\xi,r_2) : \frac{dz}{dr}(\rho)>0\right\},
\end{equation}
by (\ref{csi}), we have $z(\xi_1)>0$, moreover, if $\xi_1<r_2$, by (\ref{csi1}) we have
$$
\frac{dz}{dr}(\xi_1)=0.
$$
Differentiating (\ref{ODE2}) we get  
$$
\frac{d^2z}{dr^2}(\xi_1)>0,
$$
which gives a contradiction. Then necessarily $\xi_1=r_2$ and by continuity $\frac{dz}{dr}(r_2)\geq 0$. If $\frac{dz}{dr}(r_2)=0$, differentiating (\ref{ODE2}), we obtain again 
$$
\frac{d^2z}{dr^2}(r_2)>0,
$$ 
but this is a contradiction to $r_2=\xi_1$.
This implies $\frac{dz}{dr}(r_2)>0$ and hence the theorem is proved.
\vspace{11pt}
%PUOI TROVARE LA DIMOSTRAZIONE A PAG. 133 BLOCCO II
\section{\fontsize{10}{10}\selectfont \bf 4 - WHAT HAPPENES TO $\lambda_1$ WHEN WE PINCH THE BALL?}
We know that, if $u_1$ is the eigenfunction of problem associated to $\lambda_1(B_r)$, we have
\begin{equation}
\lambda_1(B_r) = \frac{ \displaystyle\int_{B_r} \vert \nabla u_1 \vert^2 dx - \alpha \displaystyle\int_{\partial B_r} u_1^2 d\sigma}{ \displaystyle\int_{B_r} u_1^2 dx} = \frac{\displaystyle\int_{B_r} \vert\nabla u_1\vert^2 dx - n\alpha\omega_n u_1^2(r) r^n}{\displaystyle\int_{B_r} u_1^2 dx}\mbox{.}
\end{equation}
Let $\epsilon > 0$, we consider the annulus $A_{\epsilon, r'}$, with $r'>r$ such that $\vert A_{\epsilon,r'}\vert = \vert B_r \vert$ and let $u_1$ be the function in $H^1(B_{r'})$ defined by the following statement
\begin{equation}\label{def_w}
w(x)=
\begin{cases}
u_1(x) & \mbox{if}\ x\in B_r\\
u_1(r) & \mbox{if}\ x\in B_{r'}\setminus B_r.
\end{cases}
\end{equation}
We have
\begin{equation}\label{ma}
\lambda_1(A_{\epsilon , r'})  \leq   \frac{\displaystyle\int_{A_{\epsilon , r'}}\vert \nabla w \vert^2 dx - \alpha\displaystyle\int_{\partial A_{\epsilon , r'}} w^2 d\sigma}{\displaystyle\int_{A_{\epsilon , r'}} w^2 dx}
\end{equation} 
\begin{equation*}                              =     \frac{\displaystyle\int_{B_r'}\vert\nabla w\vert^2 dx - \displaystyle\int_{B_\epsilon} \vert\nabla w\vert^2 dx -\alpha\left(\displaystyle\int_{\partial B_{r'}} w^2 d\sigma + \displaystyle\int_{\partial B_{\epsilon}} w^2d\sigma \right)}{\displaystyle\int_{B_{r'}} w^2dx - \displaystyle\int_{B_\epsilon} w^2dx}.  
\end{equation*}
We have
\begin{equation}
\int_{B_{r'}}\vert\nabla w\vert^2 dx = \int_{B_r}\vert\nabla u_1\vert^2 dx,
\end{equation} 

\begin{equation}
\int_{B_\epsilon}\vert\nabla w\vert^2 dx = o(\epsilon^n),
\end{equation}

\begin{equation}
-\alpha\left(\int_{\partial B_{r'}}w^2 d\sigma +\int_{\partial B_\epsilon}w^2 d\sigma\right)= -n\alpha\omega_n r^{n-1}u_1^2(r) - \mathcal{O}(\epsilon^{n-1}), 
\end{equation}

\begin{equation}
\int_{B_{r'}}w^2 dx - \int_{B_\epsilon}w^2 dx = \int_{B_{r}}u_1^2 dx + \mathcal{O}(\epsilon^n).
\end{equation}
From (\ref{ma}) and the above equations, we have

\begin{equation}
\lambda_1(B_r)-\lambda_1(A_{\epsilon , r'}) \geq \frac{\mathcal{O}(\epsilon^{n-1})}{\displaystyle\int_{B_r} u_1^2 dx + \mathcal{O}(\epsilon^n)}
\end{equation}
then, for $\epsilon$ small enough, we have $\lambda_1(B_r) > \lambda_1(A_{\epsilon , r'})$.
\vspace{11pt}

\section{\fontsize{10}{10}\selectfont \bf 5 - PROOF OF THEOREM 2}
The min-max formulation (\ref{minmax}) for the second eigenvalue of problem (\ref{RobinLaplacianProblem}) allows to write 
\begin{equation}\label{ineq1}
\lambda_2(\Omega)\leq\max_{v\in E_2}\frac{\displaystyle\int_\Omega{\vert\nabla v\vert}^2dx-\frac{1}{r}\displaystyle\int_{\partial\Omega}v^2d\sigma}{\displaystyle\int_\Omega v^2dx}
\end{equation}
where $E_2$ is a 2-dimensional space of the $H^1(\Omega)$. We choose $E_2$ as the subspace spanned by the coordinate function $x_i$ and a costant function, and then, denoting by $a_i\in\R$ the constant achiving the maximum in (\ref{ineq1}), we have
\begin{eqnarray}\label{fracdis}
\lambda_2(\Omega) & \leq & \frac{\displaystyle\int_\Omega\vert\nabla(x_i+a_i)\vert^2 dx-\displaystyle\frac{1}{r}\displaystyle\int_{\partial\Omega}(x_i+a_i)^2d\sigma}{\displaystyle\int_\Omega (x_i+a_i)^2dx}\\
 & = & \frac{\vert\Omega\vert -\displaystyle\frac{1}{r}\displaystyle\int_{\partial\Omega}(x_i+a_i)^2d\sigma}{\displaystyle\int_\Omega(x_i+a_i)^2dx}\mbox{.}\nonumber
\end{eqnarray}

%in particular, we obtain
%\begin{equation}\label{nofracdis}
%\lambda_2(\Omega)\int_\Omega(x_i+a_i)^2 \leq \vert\Omega\vert - \frac{1}{R}\
%\end{equation}

From (\ref{fracdis}), adding for every index, from $1$ to $d$, we obtain the following inequality

\begin{equation}\label{in_dim}
\lambda_2(\Omega) \leq \frac{d\vert\Omega\vert-\displaystyle\frac{1}{r}\displaystyle\int_{\partial\Omega}\vert x+a \vert^2 d\sigma}{\displaystyle\int_{\Omega}\vert x+a \vert^2 dx},
\end{equation}

and from that, by means of inequality (\ref{BrockMercaldo}), using a simple change of variables, we have

\begin{equation}
\lambda_2(\Omega) \leq \frac{d\vert\Omega\vert-\displaystyle\frac{1}{r}\displaystyle\int_{\partial B_r - a}\vert x+a \vert^2 d\sigma}{\displaystyle\int_{\Omega}\vert x+a \vert^2 dx} = 0 =\lambda_2(B_r),
\end{equation} 

and this completes the proof.
\section{\fontsize{10}{10}\selectfont \bf 6 - REFERENCES}
%\begin{thebibliography}{99}
\indent Abramowitz M.S. and Stegun I.A. (1965), Handbook of Mathematical Functions. Dover, New York, USA

\indent Antunes P.R.S. and Freitas P. (2012) Numerical Optimization of Low Eigenvalues of the Dirichlet and Neumann Laplacians. J. Optim. Theory Appl. {\bf 154}, 235-257

\indent Bandle C. and Wagner A. (2014) Isoperimetric inequalities for the principal eigenvalue of a membrane and the energy of problems with Robin boundary conditions. arXiv:1403.3249v2 

\indent Betta M.F., Brock F., Mercaldo A. and Posteraro M.R. (1999) A weighted isoperimetric inequality and applications to symmetrizetions. J. of Inequal. and Appl. {\bf 4}, 215-240  

\indent Brasco L., De Philippis G. and Ruffini B. (2012) Spectral oprimization for the Stekloff-Laplacian: The stability issue. Journal of Functional Analysis. {\bf 262}, 4675-4710  

\indent Brock F. (2001) An Isoperimetric Inequality for Eigenvalues of the Stekloff Problem. Z. Angew. Math. Mech. {\bf 81}, 1, 69-71

\indent Ferone V., Nitsch C. and Trombetti C. (2015) On a conjectured reverse Faber-Krahn inequality for a Steklov-tipe Laplacian eigenvalue. Comunications on Pure and Applied Analysis. {\bf 14}, 1, 63-81

\indent Freitas P. and Krej\u{c}i\u{r}\'{i}k D. (2015) The first Robin eigenvalue with negative boundary parameter. Advanced in Mathematics. {\bf 280} 322-339

\indent Henrot A. (2006) Extremum Problems for Eigenvalues of Elliptic Operators. Birkh\:{a}user Verlag, Berlin, Germany

\indent Henrot A. (2003) Minimization problems for eigenvalues of the Laplacian. Journal of Evolution Equations. {\bf 3}, 443-461 

\indent Kennedy J. (2009) An isoperimetric inequality for the second eigenvalue of the Laplacian with Robin boundary condition. Proceeding of the American Mathematical Socity. {\bf 137}, 2, 627-633

\indent Lacey A.A., Ockendon J.R. and Sabina J. (1998) Multidimensional reaction diffusion equations with nonlinear boundary conditions. SIAM J. Appl. Math. {\bf 58}, 5, 1622-1647

\indent Nitsch C. (2014) An isoperimetric result for the fundamental frequency via domain derivative. Calculus of Variation and Partial Differential Equations. {\bf 49}, 1, 323-335

\indent Pankrashkin K. and Popoff N. (2015) Mean curvature bounds and eigenvalues of Robin Laplacians. Calculus of Variation and Partial Differential Equations. {\bf 54}, 2, 1947-1961 
%\end{thebibliography}  

\end{document}